\documentclass[a4paper,centertags,oneside,11pt]{amsart}
\usepackage{mathrsfs}
\usepackage{appendix}
\usepackage{amssymb}
\usepackage{fancyhdr}
\usepackage{charter}
\usepackage{typearea}
\usepackage{pdfsync}
\usepackage{mathrsfs}
\usepackage[a4paper,top=3cm,bottom=3cm,left=2cm,right=2cm]{geometry}

\usepackage{enumitem}
\setlist[1]{itemsep=5pt}
\newcommand{\comment}[1]{}

\makeatletter
      \def\@setcopyright{}
      \def\serieslogo@{}
      \makeatother

\newcommand{\Complex}{\mathbb C}
\newcommand{\Real}{\mathbb R}

\newcommand{\ddbar}{\overline\partial}
\newcommand{\pr}{\partial}
\newcommand{\ol}{\overline}
\newcommand{\Td}{\widetilde}
\newcommand{\norm}[1]{\left\Vert#1\right\Vert}
\newcommand{\abs}[1]{\left\vert#1\right\vert}

\newcommand{\set}[1]{\left\{#1\right\}}
\newcommand{\To}{\rightarrow}

\newcommand{\cali}[1]{\mathscr{#1}}

\theoremstyle{plain}
\newtheorem{thm}{Theorem}[section]

\newtheorem{lem}[thm]{Lemma}

\theoremstyle{definition}
\newtheorem{defn}[thm]{Definition}

\theoremstyle{remark}

\newtheorem{con}[thm]{Conjecture}
\numberwithin{equation}{section}

\begin{document}
\title[]{Bergman kernel asymptotics and a pure analytic proof of Kodaira embedding theorem}
\author[]{Chin-Yu Hsiao}
\address{Institute of Mathematics, Academia Sinica, 6F, Astronomy-Mathematics Building,
No.1, Sec.4, Roosevelt Road, Taipei 10617, Taiwan}
\thanks{The author was partially supported by Taiwan Ministry of Science of Technology project 
103-2115-M-001-001 and the Golden-Jade fellowship of Kenda Foundation}
\email{chsiao@math.sinica.edu.tw or chinyu.hsiao@gmail.com}

\begin{abstract}
In this paper, we survey recent results in~\cite{HM12} about the asymptotic expansion of Bergman kernel and we give a Bergman kernel proof of Kodaira embedding theorem. 
\end{abstract}

\maketitle \tableofcontents

\section{Introduction and Set up} \label{s-I}

Let $L$ be a holomorphic line bundle over a complex manifold $M$ and let $L^k$ be the $k$-th tensor power of $L$.
The Bergman projection $P_k$ is the orthogonal projection onto the space of $L^2$-integrable holomorphic sections of $L^k$. 
The study of the large $k$ behaviour of $P_k$ is an active research subject in complex geometry and is closely related to topics like the structure of algebraic manifolds, the existence of canonical K\"ahler metrics, Toeplitz quantization, equidistribution of zeros of holomorphic sections, quantum chaos and mathematical physics.
We refer the reader to the book \cite{MM07} for a comprehensive study of the Bergman kernel and its applications and also to the survey \cite{Ma10}.

When $M$ is compact and $L$ is positive, Catlin~\cite{Cat97} and ~Zelditch~\cite{Zel98} established the asymptotic expansion of the Bergman kernel (see Theorem~\ref{t-gue140923a}) by using a fundamental result by Boutet de Monvel-Sj\"{o}strand~\cite{BouSj76} about the asymptotics of the Szeg\"{o} kernel on a strictly pseudoconvex boundary. X. Dai, K. Liu and X. Ma \cite {DLM06} obtained the full off-diagonal asymptotic expansion and Agmon estimates of the Bergman kernel for a high power of positive line bundle on a compact complex manifold by using the heat kernel method.  Ma and Marinescu~\cite{MM07}, \cite{MM08a} proved the asymptotic expansion for yet another generalization of the Kodaira Laplacian, namely the renormalized Bochner-Laplacian on a symplectic manifold and also showed the existence of the estimate on a large class of non-compact manifolds. 
Another proof based on microlocal analysis of the existence of the full asymptotic expansion for the Bergman kernel for a high power of a positive line bundle on a compact complex manifold was obtained by Berndtsson, Berman and Sj\"{o}strand~ \cite{BBS04}.

In~\cite{HM12}, we impose a very mild semiclassical local condition on $\ddbar_k$, namely the $O(k^{-N})$ small spectral gap on an open set $D\Subset M$ (see Definition~\ref{s1-d2bis}), where $\ddbar_k$ denotes the Cauchy-Riemann operator with values in $L^k$. We prove that the Bergman kernel admits an asymptotic expansion on $D$ if $\ddbar_{k}$ has $O(k^{-N})$ small spectral gap on $D$, cf.\ Theorem \ref{t-gue140923}. Our approach bases on the microlocal Hodge decomposition for Kohn Laplacian established in~\cite{Hsiao08}. The distinctive feature of these asymptotics is that they work under minimal hypotheses. This allows us to apply them in situations which were up to now out of reach. We illustrate this in the study of the Bergman kernels of positive but singular
Hermitian line bundles (see Theorem~\ref{t-gue140923II}). 

\subsection{Set up}\label{s-gue140926}

In this paper, we let $M$ be a not necessary compact complex manifold of dimension $n$ with a smooth positive $(1,1)$ form $\Theta$. $\Theta$ induces Hermitian metrics on the complexified tangent bundle $\Complex TM$ and $T^{*0,q}M$ bundle of $(0,q)$ forms of $M$ , $q=0,1,\ldots,n$. We shall denote all these Hermitian metrics  by $\langle\,\cdot\,|\,\cdot\,\rangle$.
Let $(L,h^L)\To M$ be a holomorpic line bundle over $M$, where $h^L$ denotes the Hermitian fiber metric of $L$. Let $R^L$ be the canonical curvature two form induced by $h^L$.  Given a local trivializing section $s$ of $L$ on an open subset $D\subset M$ we define the associated local weight of $h^L$ by
\begin{equation} \label{s1-e1}
\abs{s(x)}^2_{h^L}=e^{-2\phi(x)},\quad\phi\in C^\infty(D, \Real).
\end{equation}
Then $R^L|_D=2\pr\ddbar\phi$. Let $(L^k,h^{L^k})$ be the $k$-th tensor power of the line bundle $L$. If $s$ is a local
trivializing section of $L$, $\abs{s}^2_{h^L}=e^{-2\phi}$, then $s^k$ is a local trivializing
section of $L^k$ and $\abs{s^k}^2_{h^{L^k}}=e^{-2k\phi}$. 
We take $dv_M=dv_M(x)$
as the volume form on $M$ induced by $\Theta$. For every $q=0,1,2,\ldots,n$, let $(\,\cdot\,|\,\cdot\,)$ and $(\,\cdot\,|\,\cdot\,)_{h^{L^k}}$ be the standard $L^2$ inner products on $\Omega^{0,q}_0(M):=C^\infty_0(M,T^{*0,q}M)$ and $\Omega^{0,q}_0(M,L^k):=C^\infty_0(M,T^{*0,q}M\otimes L^k)$ respectively induced by $dv_M$, $\langle\,\cdot\,|\,\cdot\,\rangle$ and $h^{L^k}$ and we write $\norm{\cdot}$ and $\norm{\cdot}_{h^{L^k}}$ to denote the corresponding norms. Let
$L^2_{(0,q)}(M)$ and $L^2_{(0,q)}(M,L^k)$ be the completions of $\Omega^{0,q}_0(M)$ and $\Omega^{0,q}_0(M,L^k)$ with respect to $\norm{\cdot}$ and $\norm{\cdot}_{h^{L^k}}$ respectively.

Let $\ddbar_{k}:C^\infty(M,L^k)\To\Omega^{0,1}(M,L^k)$ be the Cauchy-Riemann operator with values in $L^k$. We extend 
$\ddbar_k$ to $L^2(M,L^k):=L^2_{(0,0)}(M,L^k)$ by $\ddbar_k:{\rm Dom\,}\ddbar_k\subset L^2(M, L^k)\To L^2_{(0,1)}(M, L^k)$,
where ${\rm Dom\,}\ddbar_k:=\{u\in L^2(M, L^k);\, \ddbar_ku\in L^2_{(0,1)}(M, L^k)\}$. Let
\[P_k:L^2(M,L^k)\To{\rm Ker\,}\ddbar_k\]
be the Bergman projection, i.e. $P_k$ is the orthogonal projection onto ${\rm Ker\,}\ddbar_k$ with respect to $(\,\cdot\,|\,\cdot\,)_{h^{L^k}}$ and let $P_k(x,y)\in C^\infty(M\times M,\mathscr L(L^k_y,L^k_x))$ be the distribution kernel of $P_k$. 

 \section{Terminology in semi-classical analysis}\label{s:prelim}
 
In this section, we collect some definitions and notations in semi-classical analysis. 

Let $B_k:L^2(M,L^k)\To L^2(M,L^k)$ be a continuous operator with smooth kernel $B_k(x,y)$. Let $s$, $s_1$ be local trivializing sections of $L$ on $D_0\Subset M$, $D_1\Subset M$ respectively, $\abs{s}^2_{h^L}=e^{-2\phi}$, $\abs{s_1}^2_{h^L}=e^{-2\phi_1}$. The localized operator (with respect to the trivializing sections $s$ and $s_1$) of $B_k$ is given by 
\begin{equation} \label{e-gue140824II}
\begin{split}
B_{k,s,s_1}:L^2_{{\rm comp\,}}(D_1)&\To L^2(D),\\
u&\To e^{-k\phi}s^{-k}B_k(s^k_1e^{k\phi_1}u).
\end{split}
\end{equation} 
and let $B_{k,s,s_1}(x,y)\in C^\infty(D\times D_1)$ be the distribution kernel of $B_{k,s,s_1}$, where 
\[L^2_{{\rm comp\,}}(D_1):=\set{v\in L^2(D_1);\, {\rm Supp\,}v\Subset D_1}.\]

Let $D$ be a local coordinate patch of $M$ and let $A_k:C^\infty_0(D)\To C^\infty(D)$ be a $k$-dependent continuous operator
with smooth kernel $A_k(x,y)$. We write $A_k\equiv0\mod O(k^{-\infty})$ (on $D$) or $A_k(x,y)\equiv0\mod O(k^{-\infty})$ (on $D$) if $A_k(x, y)$ satisfies $\abs{\pr^\alpha_x\pr^\beta_yA_k(x, y)}=O(k^{-N})$ locally uniformly
on every compact set in $D\times D$, for all multi-indices $\alpha, \beta\in\mathbb N^{2n}$ and all $N>0$. Let $B_k:L^2(M,L^k)\To L^2(M,L^k)$ be a $k$-dependent continuous operator
with smooth kernel. We write $B_k\equiv0\mod O(k^{-\infty})$ if $B_{k,s,s_1}\equiv0\mod O(k^{-\infty})$ for every local trivializing sections $s$ and $s_1$. 

\begin{defn} \label{d-gue140826}
Let $D$ be a local coordinate patch of $M$. Let $S(1;W)=S(1)$ be the set of all
$a\in C^\infty(D)$ such that for every $\alpha\in\mathbb N^{2n}$, there
exists $C_\alpha>0$, such that $\abs{\pr^\alpha_xa(x)}\leq
C_\alpha$ on $W$. If $a=a(x,k)$ depends on $k\in]1,\infty[$, we say that
$a(x,k)\in S_{{\rm loc\,}}(1;D)=S_{{\rm loc\,}}(1)$ if $\chi(x)a(x,k)$ uniformly bounded
in $S(1)$ when $k$ varies in $]1,\infty[$, for any $\chi\in
C^\infty_0(D)$. For $m\in\Real$, we put $S^m_{{\rm
loc}}(1;D)=S^m_{{\rm loc}}(1)=k^mS_{{\rm loc\,}}(1)$. If $a_j\in S^{m_j}_{{\rm
loc\,}}(1)$, $m_j\searrow-\infty$, we say that $a\sim
\sum\limits^\infty_{j=0}a_j$ (in $S^{m_0}_{{\rm loc\,}}(1)$) if
$a-\sum\limits^{N_0}_{j=0}a_j\in S^{m_{N_0+1}}_{{\rm loc\,}}(1)$ for every
$N_0$.  For a given sequence $a_j$ as above, we can always find such an asymptotic sum $a$ and $a$ is unique up to an element in $S^{-\infty}_{{\rm loc\,}}(1)=S^{-\infty}_{{\rm loc\,}}(1;W):=\bigcap_mS^m_{{\rm loc\,}}(1)$. 
\end{defn}

\section{Asymptotic expansion of Bergman kernel}\label{s-II}

Let $s$, $s_1$ be local trivializing sections of $L$ on $D_0\Subset M$, $D_1\Subset M$ respectively, $\abs{s}^2_{h^L}=e^{-2\phi}$, $\abs{s_1}^2_{h^L}=e^{-2\phi_1}$. 
Let $P_{k,s,s_1}$ be the localized operator of $P_k$ given by \eqref{e-gue140824II} and let $P_{k,s,s_1}(x,y)\in C^\infty(D\times D_1)$ be the distribution kernel of $P_{k,s,s_1}$. When $s=s_1$, $D=D_1$, we write $P_{k,s}:=P_{k,s,s_1}$, $P_{k,s}(x,y):=P_{k,s,s_1}(x,y)$. When $x=y$, $P_{k,s}(x,x)$ is independent of $s$. We write $P_k(x):=P_{k,s}(x,x)$ and we call $P_k(x)$ Bergman kernel function. Let $f_1\in C^\infty(M,L^k),\ldots,f_{d_k}\in C^\infty(M,L^k)$ be orthonormal frame for ${\rm Ker\,}\ddbar_{k}$, $d_k\in\set{0}\bigcup\mathbb N\bigcup\set{\infty}$. On $D_0$ and $D_1$, we write 
\[\begin{split}
&f_j=s^ke^{k\phi}\Td f_j,\ \ \Td f_j\in C^\infty(D),\ \ j=1,2,\ldots,d_k,\\
&f_j=s^k_1e^{k\phi_1}\hat f_j,\ \ \hat f_j\in C^\infty(D_1),\ \ j=1,2,\ldots,d_k.
\end{split}\]
We can check that
\begin{equation}\label{e-gue140923}
\begin{split}
&P_{k,s,s_1}(x,y)=\sum^{d_k}_{j=1}\Td f_j(x)\ol{\hat f_j}(y),\\
&P_{k}(x)=\sum^{d_k}_{j=1}\abs{f_j(x)}^2_{h^{L^k}}.
\end{split}
\end{equation}

We recall $O(k^{-N})$ small spectral gap property introduced in~\cite{HM12}
\begin{defn} \label{s1-d2bis}
Let $D\subset M$. We say that $\ddbar_{k}$ has \emph{$O(k^{-N})$ small spectral gap on $D$} if there exist constants $C_D>0$,  $N\in\mathbb N$, $k_0\in\mathbb N$, such that for all $k\geq k_0$ and $u\in C^\infty_0(D,L^k)$, we have  
\[\norm{(I-P_{k})u}_{h^{L^k}}\leq C_D\,k^{N}\norm{\ddbar_{k}u}_{h^{L^k}}.\]
\end{defn}

It should be mentioned that in~\cite{HM12}, we actually introduced $O(k^{-N})$ small spectral gap for Kodaira Laplacian. Note that $O(k^{-N})$ small spectral gap for $\ddbar_{k}$ implies $O(k^{-N})$ small spectral gap for Kodaira Laplacian.

One of the main results in~\cite{HM12} is the following

\begin{thm}\label{t-gue140923}
With the notations and assumptions used before, let $s$ be a local trivializing section of $L$ on an open set $D\subset M$, $\abs{s}^2_{h^L}=e^{-2\phi}$, and assume that $R^L$ is positive on $D$. Suppose that $\ddbar_{k}$ has $O(k^{-N})$ small spectral gap on $D$. Then, $\chi_1P_k\chi\equiv0\mod O(k^{-\infty})$ for every $\chi_1\in C^\infty_0(M)$, $\chi\in C^\infty_0(D)$ with ${\rm Supp\,}\chi_1\bigcap{\rm Supp\,}\chi=\emptyset$ and 
\[\mbox{$P_{k,s}(x,y)\equiv e^{ik\Psi(x,y)}b(x,y,k)\mod O(k^{-\infty})$ on $D$},\]
where $b(x,y,k)\sim\sum\limits^\infty_{j=0}b_j(x,y)k^{n-j}$ in the sense of Definition~\ref{d-gue140826}, $b_j(x,y)\in C^\infty(D\times D)$, $j=0,1,\ldots$, $b_0(x,x)=(2\pi)^{-n}\big|\det R^L(x)\big|$ and 
\begin{equation}\label{prop_psi}
\begin{split}
&\Psi(x,y)\in C^\infty(D\times D),\ \ \Psi(x,y)=-\ol\Psi(y,x)\,,\\
& \exists\, c>0:\ {\rm Im\,}\Psi\geq c\abs{x-y}^2\,,\ \Psi(x,y)=0\Leftrightarrow x=y \,,
\end{split}
\end{equation}
for any $p\in D$, take local holomorphic coordinates $z=(z_1,\ldots,z_n)$ vanishing at $p$, then near $(p,p)$,
\begin{equation}\label{e-gue140923f}
\Psi(z,w)=i(\phi(z)+\phi(w))-2i\sum\limits_{\alpha,\beta\in(\set{0}\bigcup\mathbb N)^n,\abs{\alpha}+\abs{\beta}\leq N}
\frac{\pr^{\abs{\alpha}+\abs{\beta}}\phi}{\pr z^\alpha\pr\ol z^\beta}(0)\frac{z^\alpha\ol w^\beta}{\alpha!\beta!}+O(\abs{(z,w)}^{N+1}),\ \ \forall N\in\mathbb N,
\end{equation}
where $\det R^L(x)=\lambda_1(x)\cdots\lambda_n(x)$, $\lambda_j(x)$, $j=1,\ldots,n$, are the eigenvalues of $R^L$ with respect to $\langle\,\cdot\,|\,\cdot\,\rangle$.

In particular, $P_k(x)\sim\sum\limits^\infty_{j=0}b_j(x,x)k^{n-j}$ in the sense of Definition~\ref{d-gue140826}.
\end{thm}

\subsection{Big line bundles and Shiffman conjecture}

As an application of Theorem~\ref{t-gue140923}, we will establish Bergman kernel asymptotic expansion for big line bundle and this yields yet another proof of the Shiffman conjecture. Until further notice, we assume that $M$ is compact. We recall

\begin{con}[Shiffman, 1990]
If $h^L$ is a singular Hermitian metric, smooth outside a proper analytic set $\Sigma$, $R^L>0$ in the sense of current, then $L$ is big.
\end{con}
$L$ is big if ${\rm dim\,}H^0(M,L^k)\approx k^n$, where $H^0(M,L^k)=\set{u\in C^\infty(M,L^k);\, \ddbar_ku=0}$. Ji and Shiffman~\cite{JS93} solved this conjecture. 

Now, we assume that $h^L$ is a singular Hermitian metric, smooth outside a proper analytic set $\Sigma$, $R^L>0$ in the sense of current. Consider the non-compact complex manifold $M\setminus\Sigma$. We also write $\ddbar_k$ to denote the Cauchy-Riemann operator on $M\setminus\Sigma$ with values in $L^k$. Let $P_{k,M\setminus\Sigma}$ be the associated Bergman projection on $M\setminus\Sigma$ and let $P_{k,M\setminus\Sigma}(x)$ be the associated Bergman kernel function. In~\cite{HM12}, we showed that 
\begin{thm}\label{t-gue140923I}
$\ddbar_k$ has $O(k^{-N})$ small spectral gap on every $D\Subset M\setminus\Sigma$.
\end{thm}

From Theorem~\ref{t-gue140923I} and Theorem~\ref{t-gue140923}, we deduce that
\begin{thm}\label{t-gue140923II}
$P_{k,M\setminus\Sigma}(x)\sim(2\pi)^{-n}\big|\det R^L(x)\big|k^n+b_1(x)k^{n-1}+b_2(x)k^{n-2}+\cdots$ locally uniformly on $M\setminus\Sigma$, where $b_j(x)\in C^\infty(M\setminus\Sigma)$, $j=1,2,\ldots$. 
\end{thm}

Let $\set{g_1,g_2,\ldots,g_{m_k}}$ be an orthonormal frame for $H^0(M,L^k)\bigcap L^2(M\setminus\Sigma,L^k)$. The multiplier Bergman kernel function is defined by 
\[P_{k,\cali{I}}(x):=\sum^{m_k}_{j=1}\abs{g_j(x)}^2_{h^{L^k}},\ \ x\in M\setminus\Sigma.\] 
The following result is essentially due o Skoda(see Demailly~\cite[Lemma\,7.3,\,Ch.\,VIII]{De:11}).
\begin{thm}\label{t-gue140923III}
$P_{k,M\setminus\Sigma}(x)=P_{k,\cali{I}}(x)$, $\forall x\in M\setminus\Sigma$.
\end{thm}

\begin{proof}[Proof of Shiffman conjecture]
From Theorem~\ref{t-gue140923III} and Theorem~\ref{t-gue140923II}, we establish Bergman kernel asymptotic expansion for big line bundle: 
\begin{equation}\label{e-gue140923bI}
\mbox{$P_{k,\cali{I}}(x)\sim(2\pi)^{-n}\big|\det R^L(x)\big|k^n+b_1(x)k^{n-1}+b_2(x)k^{n-2}+\cdots$ locally uniformly on $M\setminus\Sigma$},
\end{equation}
where $b_j(x)\in C^\infty(M\setminus\Sigma)$, $j=1,2,\ldots$. Let $K\Subset M\setminus\Sigma$. Note that ${\rm dim\,}H^0(M,L^k)\geq\int_KP_{k,\cali{I}}(x)dv_M(x)$. From this observation and \eqref{e-gue140923bI}, we reprove Shiffman conjecture.
\end{proof}

\section{A Bergman kernel proof of Kodaira embedding theorem}\label{s-gue140923a}

For a holomorphic line bundle $E\To M$, we say that $E$ is positive if there is a Hermitian metric $h^E$ of $E$ such that the associated curvature $R^E$ is positive definite on $M$. Let's recall Kodaira embedding theorem first. 

\begin{thm}\label{t-gue140925}
Let $M$ be a compact complex manifold. If there is a positive holomorphic line bundle $E$ over $M$, then $M$ can be holomorphic embedded into $\Complex\mathbb P^N$, for some $N\in\mathbb N$.
\end{thm}

We return to our situation and we will use the same notations as before.
By using H\"ormander's $L^2$ estimates~\cite{Hor90}, it is easy to see that if $M$ is compact and $R^L$ is positive on $M$ then $\ddbar_k$ has $O(k^{-N})$ small spectral gap on $M$. From this observation and Theorem~\ref{t-gue140923}, we deduce 

\begin{thm}\label{t-gue140923a}
Assume that $M$ is compact and $R^L$ is positive on $M$. Then, 
\begin{equation}\label{e-gue140923ab}
\chi_1P_k\chi\equiv0\mod O(k^{-\infty})
\end{equation}
for every $\chi_1\in C^\infty(M)$, $\chi\in C^\infty(M)$ with ${\rm Supp\,}\chi_1\bigcap{\rm Supp\,}\chi=\emptyset$.
Let $s$ be a local trivializing section of $L$ on an open set $D\subset M$, $\abs{s}^2_{h^L}=e^{-2\phi}$, then
\begin{equation}\label{e-gue140923abI}
\mbox{$P_{k,s}(x,y)\equiv e^{ik\Psi(x,y)}b(x,y,k)\mod O(k^{-\infty})$ on $D$},
\end{equation}
where $b(x,y,k)$ and $\Psi(x,y)$ are as in Theorem~\ref{t-gue140923}. 

In particular, 
\begin{equation}\label{e-gue140923abII}
\mbox{$P_k(x)\sim(2\pi)^{-n}\big|\det R^L(x)\big|k^n+b_1(x)k^{n-1}+b_2(x)k^{n-2}+\cdots$ uniformly on $M$}.
\end{equation}
\end{thm}

By using Theorem~\ref{t-gue140923}, we are going to give a Bergman kernel proof of Kodaira embedding theorem. 
From now on, we assume that $R^L$ is positive on $M$.
As before, put 
\[H^0(M,L^k):=\set{u\in C^\infty(M,L^k);\, \ddbar_ku=0}\] 
and let $\set{f_1,\ldots,f_{d_k}}$ be an orthonormal basis for $H^0(M,L^k)$ with respect to $(\,\cdot\,|\,\cdot\,)_{h^{L^k}}$. The Kodaira map is given by
\begin{equation}\label{e-gue140923abIII}
\Phi_k:x\in X\To[f_1(x),f_2(x),\ldots,f_{d_k}(x)]\in\Complex\mathbb P^{d_k-1}.
\end{equation}
From \eqref{e-gue140923abII}, we see that there is a $k_0>0$ such that for every $k\geq k_0$, $\sum\limits^{d_k}_{j=1}\abs{f_j(x)}^2_{h^{L^k}}\geq ck^n$ on $M$, where $c>0$ is a constant independent of $k$. Hence, fix any $k\geq k_0$, for every $x\in X$, there is a $f_j$, $j\in\set{1,2,\ldots,d_k}$, such that $\abs{f_j(x)}^2_{h^{L^k}}>0$. We conclude that $\Phi_k$ is a well-defined as a smooth map from $X$ to $\Complex\mathbb P^{d_k-1}$. We will prove 

\begin{thm}\label{t-gue140923ab}
For $k$ large, $\Phi_k$ is a holomorphic embedding.
\end{thm}

It is clearly that Kodaira embedding theorem follows from Theorem~\ref{t-gue140923ab}. We recall that for a smooth map $\Phi:X\To\Complex\mathbb P^{N}$ is an embedding if $d\Phi_x:TX\To T\Complex\mathbb P^N$ is injective at each point $x\in X$ and $\Phi:X\To\Complex\mathbb P^{N}$ is globally injective. 

Let $s$ be a local trivializing section of $L$ on an open set $D\subset M$. Fix $p\in D$ and let $z=(z_1,\ldots,z_n)=x=(x_1,\ldots,x_{2n})$, $z_j=x_{2j-1}+ix_{2j}$, $j=1,\ldots,n$, be local holomorphic coordinates of $X$ defined in some small neighbourhood of $p$ such that 
\begin{equation}\label{e-gue140923fI}
\phi(z)=\sum\limits^n_{j=1}\lambda_j\abs{z_j}^2+O(\abs{z}^3),
\end{equation}
where $2\lambda_1,\ldots,2\lambda_n$ are the eigenvalues of $R^L(p)$ with respect to $\langle\,\cdot\,|\,\cdot\rangle$. We may assume that the local coordinates $z$ defined on $D$. We also write $y=(y_1,\ldots,y_{2n})$. Until further notice, we work on $D$. Take $\chi\in C^\infty_0(\Real,[0,1])$ with $\chi(x)=1$ on $[-\frac{1}{2},\frac{1}{2}]$, $\chi(x)=0$ on $]-\infty,-1]\bigcup[1,\infty[$ and $\chi(t)=\chi(-t)$, for every $t\in\Real$. Let 
\begin{equation}\label{e-gue140923fII}
u_{k}:=P_k\Bigr(s^ke^{k\phi}\chi(\sqrt{k}y_1)\cdots\chi(\sqrt{k}y_{2n})\Bigr)\in H^0(M,L^k).
\end{equation}
On $D$, we write $u_{k}=s^ke^{k\phi}\Td u_{k}$, $\Td u_{k}\in C^\infty(D)$. Then, $\abs{u_{k}(x)}^2_{h^{L^k}}=\abs{\Td u_{k}(x)}^2$, $\forall x\in D$. We need

\begin{lem}\label{l-gue140924}
With the notations used above, there is a $k_0>0$ independent of $k$ and the point $p$ such that for all $k\geq k_0$,
\begin{equation}\label{e-gue140924}
\abs{u_k(p)}^2_{h^{L^k}}\geq c_0,
\end{equation}
\begin{equation}\label{e-gue140924I}
\abs{u_k(x)}^2_{h^{L^k}}\leq\frac{1}{c_0k},\ \ \forall x\notin D
\end{equation}
and
\begin{equation}\label{e-gue140924ba}
\abs{\frac{1}{\sqrt{k}}\frac{\pr\Td u_{k}}{\pr x_s}(p)}\leq\frac{1}{c_0k},\ \ s=1,2,\ldots,2n,
\end{equation}
where $c_0>0$ is a constant independent of $k$ and the point $p$.
\end{lem}

\begin{proof}
From \eqref{e-gue140923abI}, we can check that
\begin{equation}\label{e-gue140924II}
\begin{split}
&\Td u_{k}(x)\\
&\equiv\int e^{ik\Psi(x,y)}b(x,y,k)\chi(\sqrt{k}y_1)\cdots\chi(\sqrt{k}y_{2n})dv_M(y)\mod O(k^{-\infty})\\
&\equiv\int e^{ik\Psi(x,\frac{y}{\sqrt{k}})}k^{-n}b(x,\frac{y}{\sqrt{k}},k)\chi(y_1)\cdots\chi(y_{2n})dv_M(y)\mod O(k^{-\infty}).
\end{split}
\end{equation}
From \eqref{e-gue140923f}, Theorem~\ref{t-gue140923a} and note that $\Psi(0,0)=0$, we can check that 
\[\lim_{k\To\infty}\Td u_{k}(p)=\frac{1}{2}\pi^{-n}\abs{\det R^L_p}\int\chi(y_1)\cdots\chi(y_{2n})dv_M(y).\]
Similarly, it is straightforward to check that $\lim_{k\To\infty}\frac{1}{\sqrt{k}}\frac{\pr\Td u_{k}}{\pr x_s}(p)=0$, $s=1,2,\ldots,2n$. 
Hence, there is a constant $\Td k_0>0$ such that for every $k\geq\Td k_0$, \eqref{e-gue140924} and \eqref{e-gue140924ba} hold. Since $X$ is compact, $\Td k_0$ can be taken to be independent of the point $p$.

Now, we prove \eqref{e-gue140924I}. Since $x\notin D$, from \eqref{e-gue140923ab}, we see that $\abs{u_k(x)}^2_{h^{L^k}}\equiv 0\mod O(k^{-\infty})$ outside $D$. Thus, there is a constant $\hat k_0>0$ such that for every $k\geq\hat k_0$, \eqref{e-gue140924I} holds. Since $X$ is compact, $\hat k_0$ can be taken to be independent of the point $p$. The lemma follows.
\end{proof}

For every $j=1,2,\ldots,n$, let 
\begin{equation}\label{e-gue140924IV}
u^j_{k}:=P_k\Bigr(s^ke^{k\phi}\sqrt{k}(y_{2j-1}+iy_{2j})\chi(\sqrt{k}y_1)\cdots\chi(\sqrt{k}y_{2n})\Bigr)\in H^0(M,L^k).
\end{equation}
On $D$, we write $u^j_{k}=s^ke^{k\phi}\Td u^j_{k}$, $\Td u^j_{k}\in C^\infty(D)$, $j=1,2,\ldots,n$. The following follows from some straightforward computation and essentially the same as the proof of Lemma~\ref{l-gue140924}. We omit the details.

\begin{lem}\label{l-gue140924I}
With the notations used above, there is a $k_1>0$ independent of $k$ and the point $p$ such that for all $k\geq k_1$, 
\begin{equation}\label{e-gue140925}
\begin{split}
&\abs{\Td u^j_{k}(p)}\leq\frac{1}{c_1k},\ \ j=1,2,\ldots,n,\ \ 
\abs{\frac{1}{\sqrt{k}}\frac{\pr\Td u^j_{k}}{\pr\ol z_{s}}(p)}\leq\frac{1}{c_1k},\ \ j,s=1,2,\ldots,n,\\
&\abs{\frac{1}{\sqrt{k}}\frac{\pr\Td u^j_{k}}{\pr z_{s}}(p)}\leq\frac{1}{c_1k},\ \ j,s=1,2,\ldots,n-1,\ \ j\neq s,\ \ 
\abs{\frac{1}{\sqrt{k}}\frac{\pr\Td u^j_{k}}{\pr z_j}(p)}\geq c_1,\ \ j=1,2,\ldots,n,
\end{split}
\end{equation}
where $c_1>0$ is a constant independent of $k$ and the point $p$.
\end{lem}

From now on, we take $k$ be a large constant so that $k>>2(k_0+k_1)$, where $k_0>0$ and $k_1>0$ are constants as in Lemma~\ref{l-gue140924} and Lemma~\ref{l-gue140924I}. We can prove 

\begin{thm}\label{t-gue140924}
$d\Phi_k(x):T_xX\To T_x\Complex\mathbb P^{d_k-1}$ is injective at every $x\in X$.
\end{thm}

\begin{proof}
Fix $p\in X$ and let $s$ be a local trivializing section of $L$ on an open set $D\subset M$, $p\in D$. Let $u_k\in H^0(M,L^k)$ and $u^j_k\in H^0(M,L^k)$, $j=1,2,\ldots,n$, be as in Lemma~\ref{l-gue140924} and Lemma~\ref{l-gue140924I}. From Lemma~\ref{l-gue140924} and Lemma~\ref{l-gue140924I}, it is not difficult to check that $u_k,u^1_k,u^2_k\ldots,u^n_k$ are linearly independent. Take $\set{u_k,u^1_k,u^2_k,\ldots,u^n_k,g_1,\ldots,g_{m_k}}$ be a basis (not orthogonal) for $H^0(M,L^k)$, $m_k=d_k-n-1$. From Lemma~\ref{l-gue140924} and Lemma~\ref{l-gue140924I}, it is easy to see that 
\begin{equation}\label{e-gue140924VIII}
\mbox{the differential of the map $x\To (\frac{u^1_k}{u_k},\ldots,\frac{u^n_k}{u_k},\frac{g_1}{u_k},\ldots,\frac{g_{m_k}}{u_k})$ is injective at $p$}.
\end{equation}
From \eqref{e-gue140924VIII} and some elementary linear algebra, it is not difficult to check that $d\Phi_k(p):T_pX\To T_p\Complex\mathbb P^{d_k-1}$ is injective. We omit the detail.
\end{proof}

Now, we can prove

\begin{thm}\label{t-gue140924I}
For $k$ large, $\Phi_k: X\To\Complex\mathbb P^{d_k-1}$ is globally injective.
\end{thm}

\begin{proof}
We assume that the claim of the theorem is not true. We can find $x_{k_j}, y_{k_j}\in M$, $x_{k_j}\neq y_{k_j}$, $0<k_1<k_2<\cdots$, $\lim_{j\To\infty}k_j=\infty$, such that $\Phi_{k_j}(x_{k_j})=\Phi_{k_j}(y_{k_j})$, for each $j$. We may suppose that there are $x_{k}, y_k\in M$, $x_k\neq y_{k}$, such that $\Phi_{k}(x_k)=\Phi_{k}(y_k)$, for each $k$. Thus,
$[f_1(x_k),\cdots,f_{d_k}(x_k)]=[f_1(y_k),\cdots,f_{d_k}(y_k)]$, for each $k$. We conclude that for every $g_k\in H^0(M,L^k)$, there is a $\lambda_k\in\Complex$ such that
\begin{equation}\label{e-gue140924f}
g_k(x_k)=\lambda_kg_k(y_k).
\end{equation}
We may assume that $\abs{\lambda_k}\geq1$. Hence, for every $g_k\in H^0(M,L^k)$,
\begin{equation}\label{e-gue140924fI}
\abs{g_k(x_k)}^2_{h^{L^k}}\geq\abs{g_k(y_k)}^2_{h^{L^k}}.
\end{equation}

Since $M$ is compact, we may assume that $x_k\To p\in M$, $y_k\To q\in M$, as $k\To\infty$.  Suppose that $p\neq q$. In view of Lemma~\ref{l-gue140924}, we see that there is a $v_k\in H^0(M,L^k)$ with $\abs{v_k(y_k)}^2_{h^{L^k}}\geq c_0$ and $\abs{v_k(x_k)}^2_{h^{L^k}}\leq\frac{1}{c_0k}$, where $c_0>0$ is a constant independent of $k$. Thus, for $k$ large, $\abs{v_k(x_k)}^2_{h^{L^k}}<\abs{v_k(y_k)}^2_{h^{L^k}}$. From this and \eqref{e-gue140924fI}, we get a contradiction. Thus, we must have $p=q$.

Let $s$ be a local trivializing section of $L$ on an open subset $D\subset X$ of $p$, $\abs{s}^2_{h^L}=e^{-2\phi}$.
Now, we assume that $x_k\To p\in M$, $y_k\To p\in M$, as $k\To\infty$. Let $z=(z_1,\ldots,z_n)=x=(x_1,\ldots,x_{2n})$, $z_j=x_{2j-1}+ix_{2j}$, $j=1,\ldots,n$, be local holomorphic coordinates of $X$ defined in some small neighbourhood of $p$ such that \eqref{e-gue140923fI} hold. We may assume that $x_k, y_k\in D$ for each $k$ and the local coordinates $x$ defined on $D$. We shall use the same notations as before. 

{\rm Case I\,}: $\limsup_{k\To\infty}\sqrt{k}\abs{x_k-y_k}=M>0$ ($M$ can be $\infty$). \\
For simplicity, we may assume that 
\begin{equation}\label{e-gue140923fIIa}
\lim_{k\To\infty}\sqrt{k}\abs{x_k-y_k}=M,\ \ M\in]0,\infty].
\end{equation}
On $D$, we write $f_j=s^ke^{k\phi}\Td f_j$, $\Td f_j\in C^\infty(D)$, $j=1,\ldots,d_k$. Put
\begin{equation}\label{e-gue140923fIII}
v_k(x):=\sum\limits^{d_k}_{j=1}f_j(x)\ol{\Td f_j(y_k)}\in H^0(M,L^k).
\end{equation}
We can check that 
\begin{equation}\label{e-gue140923fIV}
\begin{split}
\abs{v_k(x_k)}^2_{h^{L^k}}&=\abs{\sum\limits^{d_k}_{j=1}\Td f_j(x_k)\ol{\Td f_j(y_k)}}^2=\abs{P_{k,s}(x_k,y_k)}^2=\abs{e^{ik\Psi(x_k,y_k)}b(x_k,y_k,k)}^2\\
&\leq e^{-2k{\rm Im\,}\Psi(x_k,y_k)}\abs{b(x_k,y_k,k)}^2
\end{split}
\end{equation}
and 
\begin{equation}\label{e-gue140923fV}
\abs{v_k(y_k)}^2_{h^{L^k}}=\abs{P_{k,s}(y_k,y_k)}^2=\abs{e^{ik\Psi(y_k,y_k)}b(y_k,y_k,k)}^2=\abs{b(y_k,y_k,k)}^2.
\end{equation}
From the fact that ${\rm Im\,}\Psi(x,y)\geq c\abs{x-y}^2$, where $c>0$ is a constant, \eqref{e-gue140923fIIa}, \eqref{e-gue140923fIV} and \eqref{e-gue140923fV}, we can check that 
\begin{equation}\label{e-gue140923fVI}
\lim_{k\To\infty}k^{-2n}\abs{v_k(x_k)}^2_{h^{L^k}}\leq e^{-2cM^2}\abs{b_0(p,p)}^2<\abs{b_0(p,p)}^2=\lim_{k\To\infty}k^{-2n}\abs{v_k(y_k)}^2_{h^{L^k}},
\end{equation}
where $b_0$ is the leading term of $b(x,y,k)$. Note that $b_0(p,p)=(2\pi)^{-n}\big|\det R^L(p)\big|>0$ (see Theorem~\ref{t-gue140923}). From \eqref{e-gue140923fVI} and \eqref{e-gue140924fI}, we get a contradiction.

{\rm Case II\,}: $\limsup_{k\To\infty}\sqrt{k}\abs{x_k-y_k}=0$.

Put $f_k(t)=\frac{\abs{v_k(tx_k+(1-t)y_k)}^2_{h^{L^k}}}{P_k(tx_k+(1-t)y_k)P_k(y_k)}$, where $v_k$ is as in \eqref{e-gue140923fIII}. We can check that 
\begin{equation}\label{e-gue140923fVII}
f_k(t)=\frac{\abs{\sum\limits^{d_k}_{j=1}\Td f_j(tx_k+(1-t)y_k)\ol{\Td f_j(y_k)}}^2}{\sum\limits^{d_k}_{j=1}\abs{\Td f_j(tx_k+(1-t)y_k)}^2\sum\limits^{d_k}_{j=1}\abs{\Td f_j(y_k)}^2}=\frac{\abs{P_{k,s}(tx_k+(1-t)y_k,y_k)}^2}{P_k(tx_k+(1-t)y_k)P_k(y_k)}.
\end{equation}
From \eqref{e-gue140924f} and \eqref{e-gue140923fVII}, it is easy to see that $0\leq f_k(t)\leq 1$, $\forall t\in[0,1]$ and $f_k(0)=f_k(1)=1$. Thus, for each $k$, there is a $t_k\in[0,1]$ such that $f''_k(t_k)\geq0$. Hence, 
\begin{equation}\label{e-gue140923fVIII}
\liminf_{k\To\infty}\frac{f''_k(t_k)}{\abs{x_k-y_k}^2k}\geq0.
\end{equation}
From \eqref{e-gue140923abI}, we see that 
\begin{equation}\label{e-gue140924fa}\begin{split}
&\abs{P_{k,s}(tx_k+(1-t)y_k,y_k)}^2=e^{-2k{\rm Im\,}\Psi(tx_k+(1-t)y_k,y_k)}\abs{b(tx_k+(1-t)y_k,y_k,k)}^2,\\
&P_k(tx_k+(1-t)y_k)\\
&=b(tx_k+(1-t)y_k,tx_k+(1-t)y_k,k)\sim\sum\limits^\infty_{j=0}k^{n-j}b_j(tx_k+(1-t)y_k,tx_k+(1-t)y_k).
\end{split}\end{equation}
From \eqref{e-gue140924fa}, it is straightforward to calculate that
\begin{equation}\label{e-gue140924faI}
\begin{split}
&\frac{\pr\abs{P_{k,s}(tx_k+(1-t)y_k,y_k)}}{\pr t}\\
&=e^{-2k{\rm Im\,}\Psi(tx_k+(1-t)y_k,y_k)}\Bigr(\langle\,-2k{\rm Im\,}\Psi'_x(tx_k+(1-t)y_k,y_k)\,,\,x_k-y_k\,\rangle\abs{b(tx_k+(1-t)y_k,y_k,k)}^2\\
&\quad+\langle\,O(k^{2n})\,,\,x_k-y_k\,\rangle\Bigr),\\
&\frac{\pr^2\abs{P_{k,s}(tx_k+(1-t)y_k,y_k)}}{\pr t^2}\\
&=e^{-2k{\rm Im\,}\Psi(tx_k+(1-t)y_k,y_k)}\Bigr(\bigr(\langle\,-2k{\rm Im\,}\Psi'_x(tx_k+(1-t)y_k,y_k)\,,\,x_k-y_k\,\rangle\bigr)^2\abs{b(tx_k+(1-t)y_k,y_k,k)}^2\\
&\quad+\langle\,-2k{\rm Im\,}\Psi''_x(tx_k+(1-t)y_k,y_k)(x_k-y_k)\,,\,x_k-y_k\,\rangle\abs{b(tx_k+(1-t)y_k,y_k,k)}^2\\
&\quad+\langle\,-2k{\rm Im\,}\Psi'_x(tx_k+(1-t)y_k,y_k)\,,\,x_k-y_k\,\rangle\langle\,O(k^{2n})\,,\,x_k-y_k\,\rangle+\langle\,O(k^{2n})(x_k-y_k)\,,\,x_k-y_k\,\rangle\Bigr),\\
&\frac{\pr P_{k}(tx_k+(1-t)y_k,y_k)}{\pr t}=\langle\,O(k^{2n})\,,\,x_k-y_k\,\rangle,\\ 
&\frac{\pr^2P_{k}(tx_k+(1-t)y_k,y_k)}{\pr t^2}=\langle\,O(k^{2n})(x_k-y_k)\,,\,x_k-y_k\,\rangle,
\end{split}
\end{equation}
where ${\rm Im\,}\Psi'_x(x,y)$ and ${\rm Im\,}\Psi''_x(x,y)$ denote the derivative and the Hessian of ${\rm Im\,}\Psi(x,y)$ with respect to $x$ respectively. Note that 
\[\mbox{$\abs{\langle\,-2k{\rm Im\,}\Psi'_x(tx_k+(1-t)y_k,y_k)\,,\,x_k-y_k\,\rangle}\leq\frac{1}{c_0}k\abs{x_k-y_k}^2\To0$ as $k\To\infty$}\] 
and 
\[\langle\,-2k{\rm Im\,}\Psi''_x(tx_k+(1-t)y_k,y_k)(x_k-y_k)\,,\,x_k-y_k\,\rangle<-c_0k\abs{x_k-y_k}^2,\] 
where $c_0>0$ is a constant independent of $k$. From this observation, \eqref{e-gue140923fVII} and \eqref{e-gue140924faI}, it is straightforward to see that $\liminf_{k\To\infty}\frac{f''_k(t_k)}{\abs{x_k-y_k}^2k}<0$. From this and \eqref{e-gue140923fVIII}, we get a contradiction. 

The theorem follows.
\end{proof}

From Theorem~\ref{t-gue140924} and Theorem~\ref{t-gue140924I}, we obtain Theorem~\ref{t-gue140923ab} and Kodaira embedding theorem follows then.

\end{document}